\documentclass[12pt,reqno]{amsart}
\usepackage{graphicx}
\usepackage{amssymb}

\numberwithin{equation}{section}
\newtheorem{Theorem}{Theorem}[section]
\newtheorem{Lemma}[Theorem]{Lemma}
\newtheorem{Corollary}[Theorem]{Corollary}

\newtheorem{Definition}[Theorem]{Definition}

\numberwithin{equation}{section}
\def\ab{{\bold a}}

\def\frk{\frak}               

\def\mm{{\frk m}}

\def\Phi{{\frk n}}
\def\Phi{{\frk N}}
\def\Implies{\ifmmode\Longrightarrow \else
        \unskip${}\Longrightarrow{}$\ignorespaces\fi}
\def\implies{\ifmmode\Rightarrow \else
        \unskip${}\Rightarrow{}$\ignorespaces\fi}
\def\iff{\ifmmode\Longleftrightarrow \else
        \unskip${}\Longleftrightarrow{}$\ignorespaces\fi}

\newcommand\Ass{\operatorname{Ass}}
\newcommand\supp{\operatorname{Supp}}

\newcommand\depth{\operatorname{depth}}

\newcommand\Union{\operatorname{\bigcup}}

\begin{document}

\title [Stable set of associated prime ideals] {The stable set of associated prime ideals of a squarefree principal Borel ideal}
\author [A. Aslam] {Adnan Aslam}

\address{Adnan Aslam, Abdus Salam School of Mathematical Sciences GC University, Lahore.} \email{adnanaslam15@yahoo.com}


\begin{abstract}
It is shown that a squarefree principal Borel ideal satisfies the persistence property for the associated prime ideals. For the graded maximal ideal we compute the index of stability with respect to squarefree principal Borel ideals and determine their stable set of associated prime ideals.
\end{abstract}
\maketitle

\section*{Introduction}

In recent years the stable set of monomial ideals has been studied in various papers, see for example \cite{BHR}, \cite{HRV}, \cite{MMV}, \cite{MV} and \cite{VT}.  By a result of Brodmann \cite{Br2}, for  any graded ideal $I$ in the polynomial ring $S$ there exists an integer $k_0$ such that  $\Ass(I^k)=\Ass(I^{k+1})$ for $k\geq k_0$. The smallest integer $k_0$ with this property is called the index of stability of $I$ and $\Ass(I^{k_0})$ is  called the set of stable prime ideals of $I$. A  prime ideal $P\in \Union_{k\geq 1}\Ass(I^k)$ is said to be persistent  with respect to $I$ if whenever $P\in \Ass(I^k)$ then $P\in \Ass(I^{k+1})$, and the ideal $I$ is said to satisfy the persistence property  if all prime ideals $P\in \Union_{k\geq 1}\Ass(I^k)$ are persistent. It is an open question whether all squarefree monomial ideals satisfy the persistence property.

In this paper we show that any squarefree principal Borel ideal satisfies  the persistence property. The strategy of the proof is the same as that applied in \cite{HRV}. Indeed,   by using a result of De Negri \cite{DN} it follows  that all powers of a squarefree principal Borel ideal have linear resolutions. This fact together with the result that all monomial localizations of a  squarefree principal Borel ideal are again  squarefree principal Borel ideals, as shown in Theorem~\ref{localization}, we conclude  in Section 1 that this class  of ideals  satisfies the persistence property, see Corollary~\ref{persistence}.

In Section 2 we answer the question for which squarefree principal Borel ideals the graded maximal ideal of $S$ is a stable associated prime ideal. The answer is given in Theorem~\ref{important}. For the proof of this result we use   a formula about linear quotients from \cite[Section 2]{BEOS}. Then in the following Theorem~\ref{lamda1} we give a precise formula for the index of stability of the graded maximal ideal with respect to the given squarefree principal Borel ideal. As a consequence of this formula we show in Corollary~\ref{bound} that this  index stability is bounded above by the degree of the generators of the Borel ideal. Finally in the last Section~3 we describe in Theorem~\ref{main} the stable set of associated prime ideals for any  squarefree principal Borel ideal and conclude the paper with an explicit example demonstrating this theorem.

\section{The persistence property for squarefree principal Borel ideals}
Let $K$ be a field and $S=K[x_1,\ldots,x_n]$  the polynomial ring over $K$ in $n$ indeterminates. A monomial $u=x^\ab= x_1^{a_1}x_2^{a_2}\cdots x_n^{a_n}$  of degree $d$ we often write in the form $u=x_{i_1}x_{i_2}\cdots x_{i_r}$ with $1\leq i_1\leq i_2\leq \cdots \leq i_r \leq n$. The monomial $u$ is squarefree, if $\deg_{x_i}(u)\leq 1$ for all $i$, or equivalently, if $i_1<i_2<\cdots <i_r$.

We  set $\deg_{x_i}(u)=a_i$ for $i=1,\ldots,n$, and set $\min(u)=\min\{i\:\; a_i\neq 0\}$ and $\max(u)=\max\{i\:\; a_i\neq 0\}$.

For a monomial  ideal $I$   denote by $G(I)$ the unique minimal set of monomial generators of $I$.

We first recall some concepts related to stable ideals. Let $I$ be a monomial ideal in $S$ generated in single degree. Then $I$ is called a {\em $k$-strongly stable ideal}, if  all monomials $u\in G(I)$ are of the form $x_1^{a_1}\cdots x_n^{a_n}$ with $a_i\leq k$,  and for all $u\in G(I)$ and all integers  $1\leq i<j\leq n$ such that $x_j$ divides $u$ and $\deg_{x_i}(u)<k$ we have  $x_i(u/x_j)\in G(I)$. A monomial ideal is called a squarefree strongly stable ideal if it is a $1$-strongly stable ideal.

Let $I$ be a $k$-strongly stable ideal, then the elements $u_1,\ldots,u_r \in G(I)$ are called {\em Borel generators} of $I$, if $I$ is the smallest $k$-strongly stable ideal containing $u_1,\ldots,u_r$. If $u_1,\ldots,u_r$ are  Borel generators of $I$, we  set $B_k(u_1,\ldots,u_r)=G(I)$.

\begin{Definition}
{\em A monomial ideal is called $k$-strongly stable {\em principal} ideal, if there exist a monomial $u\in G(I)$ such that $I=(B_k(u))$}.
\end{Definition}

Let $u=x_{i_1}x_{i_2}\cdots x_{i_d}$ with $1\leq i_1 < i_2< \cdots < i_d\leq n$ be a squarefree monomial in $S$. We observe that
\[
B_1(u)=\{x_{j_1}\cdots x_{j_d}\:\,1\leq j_1 < j_2< \cdots < j_d\leq n \quad \text{and} \quad j_k\leq i_k \quad \text{for} \quad 1\leq k \leq d \}
\]

Let $P$ be a monomial prime ideal in $S$. Then $P$ is generated by a subset of $\{x_1,\ldots,x_n\}$, thus there exist a set $A\subset [n]$ such that $P=P_A$, where $P_A=\{x_i\:\, i\notin A\}$.

Now let $I$ be arbitrary monomial ideal in $S$. Then there exist a unique monomial ideal $I(P_A)\subset S_A$ where $S_A=K[x_i\:\, i\in A]$ such that
\[
I(P_A)S_{P_A}=IS_{P_A}.
\]

Notice that $I(P_A)$ may also be viewed as saturation of $I$ with respect to $u=\prod_{i\in A}x_i$. In other words
\[
I(P_A)=I:u^{\infty}=\bigcup_{k\geq 0} I:u^k.
\]
The ideal $I(P_A)$ is called the {\em monomial localization} of $I$ with respect to the monomial prime ideal $P_A$.

Observe that for any two subsets $A\subset B$ we have
\begin{eqnarray}
\label{double}
I(P_A)(P_{B\setminus A})=I(P_B).
\end{eqnarray}
\begin{Theorem}
\label{localization}
Let $I\subset S$ be a squarefree strongly stable principal ideal, and $A$ a subset of $[n]$. Then $I(P_A)$ is  squarefree strongly stable principal ideal in $S_A$.
\end{Theorem}
\begin{proof}

Let $u=x_{i_1}\cdots x_{i_d}$ be the Borel generator of $I$, and let $k\in[n]$. We first show that  $I(P_{\{k\}})$ is again a squarefree strongly stable monomial ideal with Borel generator $v$ where
\begin{eqnarray}
 \label{local}
 v=\left\{
  \begin{array}{ll}
    u, & \hbox{if \,  $k>\max(u)$;} \\
    u/x_{i_j}, & \hbox{if \,  $i_{j-1}<k\leq i_j$.}
  \end{array}
\right.
\end{eqnarray}
 Indeed, we claim that

 \begin{eqnarray}
 \label{formula}
 I(P_{\{k\}})=(B_1(v)).
 \end{eqnarray}

First we show that $I(P_{\{k\}})$ is a strongly stable ideal in $S_{\{k\}}$ generated either in degree $d-1$ or in degree $d$. Observe that $I(P_{\{k\}})$ is generated by the monomials $w\in G(I)$ such that $x_k\nmid w$ together with monomials $w \in S_{\{k\}}$ such that $wx_k\in G(I)$. It follows that $G(I)$ has generators at most in degree $d-1$ and in degree $d$, and generated only in degree $d$ if no element in $G(I)$ is divisible by $x_k$.

Suppose first that $k>i_d$. Then $x_k$ divides no element in $G(I)$, and hence $I(P_{\{k\}})=(B_1(u))=(B_1(v))$, and we are done.

Next suppose that $k\leq i_d$. We show that each $w\in G(I)$ which is not divisible by $x_k$ is a multiple of a generator of of degree $d-1$ in $I(P_{\{k\}})$. Let $w=x_{j_1}\cdots x_{j_d}$ be a monomial in $B_1(u)$, then $j_1\leq i_1,\ldots,j_d\leq i_d$, and suppose that $x_k$ does not divide $w$. If $k>j_d$ we set $w'=x_k(w/x_{j_d}$. Otherwise there exists an integer $t$ such that $j_{t-1}<k<j_t$, and we set $w'=x_k(w/x_{j_t}$. In both cases, $w'=x_k(w/x_{i_j})$ belongs to $I$ and $x_k$ divides $w'$. Therefore $w'/x_k$ is an element of degree $d-1$ in $I(P_{\{k\}})$, and $w=x_{i_j}(w'/x_k)$. In conclusion we find that $I(P_{\{k\}})$ is generated in degree $d-1$ if $k\leq i_d$.

It remains to be shown that $(B_1(v)) I(P_{\{k\}})$ if $k\leq i_d$. We first prove that $(B_1(v))\subset I(P_{\{k\}})$. Since  both ideals $I(P_{\{k\}})$ and $(B_1(v))$ are generated in degree $d-1$ in the case that $k\leq i_d$, and since $v\in I(P_{\{k\}})$, the inclusion $(B_1(v))\subset I(P_{\{k\}})$ will follow, once we have shown that $I(P_{\{k\}})$ is squarefree strongly stable. To see this, let $w\in G(I(P_{\{k\}}))$, then $wx_k\in G(I)$. Let $1\leq i<j\leq n$ with $i,j\neq k$ and such that $x_j|w$ and $x_i\nmid w$. It follows that $x_j|wx_k$ and $x_i\nmid wx_k$. Since $I$ is a squarefree strongly stable ideal, we have that $x_i(wx_k)/x_j=(x_iw/x_j)x_k \in I$, and hence $x_iw/x_j\in I(P_{\{k\}})$. This shows that $I(P_{\{k\}})$ is indeed squarefree strongly stable.

Next we show that $I(P_{\{k\}})\subset (B_1(v))$, let $w=x_{l_1}\cdots x_{l_{d-1}}\in G(I(P_{\{k\}}))$. Then $wx_k=x_{l_1}\cdots x_{l_{t-1}}x_kx_{l_t}\cdots x_{l_{d-1}}\in G(I)$, where $l_{t-1}<k\leq l_t$ for some $t$. As $I$ is a squarefree strongly stable principal ideal with Borel generator $u=x_{i_1}\cdots x_{i_d}$, it follows that $l_1\leq i_1,\ldots,l_{t-1}\leq i_{t-1},k\leq i_t,l_t\leq i_{t+1}.\ldots,l_{d-1}\leq i_d$. To show that $v$ is the unique Borel generator of $I(P_{\{k\}})$, it suffices to show that
\begin{eqnarray}
\label{ineq}
l_1\leq i_1,\ldots,l_{j-1}\leq i_{j-1},l_j\leq i_{j+1}.\ldots,l_{d-1}\leq i_d.
\end{eqnarray}
 Suppose $t<j$, then $i_{j-1}<k\leq i_t\leq i_{j-1}$, a contradiction. Therefore $t\geq j$, and the inequalities (\ref{ineq}) are satisfied and this proves our claim.

 Finally we consider the general case. Let $A=\{{j_1},\ldots ,{j_r}\}\subset [n]$, and set  $B=\{{j_1},\ldots ,{j_{r-1}}\}$. We proceed by induction on $r$. If $r=1$, then the statement follows from what we have already shown. Let $B=\{{j_1},\ldots ,{j_{r-1}}\}$. Then $B\subset A$ and $I(P_B)(P_{\{j_r\}})=I(P_A)$. By our induction hypothesis, $I(P_B)$ is a squarefree strongly stable principal ideal in $S_B$. Hence the desired conclusion follows from the case discussed in the first part of the proof.
\end{proof}

 A prime ideal $P\in V(I)$  is said to be a {\em persistent prime ideal} of $I$, if whenever $P\in \Ass(I^k)$ for some exponent $k$, then $P\in \Ass(I^{k+1})$. If this happens to be so for $k$, then of course we have $P\in \Ass(I^{\ell})$ for all $\ell\geq k$. The ideal $I$ is said to have the {\em persistence property} if all prime ideals $P\in \Union_k\Ass(I^k)$ are persistent prime ideals.
\begin{Corollary}
\label{persistence}
Let $I$ be squarefree strongly stable principal ideal. Then $I$ satisfies the persistence property.
\end{Corollary}
\begin{proof}
Let $P\in \Ass(I^k)$ for some $k>0$. Since $I$ is a monomial ideal it follows that $P$ is a monomial prime ideal. Therefore there exist a subset $A\subset [n]$ such that $P=P_A$. One has $P_A\in \Ass(I^k)$ if and only if $\mm \in \Ass(I(P_A)^k)$, where $\mm$ is the graded maximal ideal of $S_A$. We know that $\mm \in \Ass({I(P_A)}^k)$ if and only if $\depth({I(P_A)}^k)=0$. Since $I$ is  a squarefree strongly stable principal ideal, Theorem \ref{localization} implies that ${I(P)}$ is again a squarefree strongly stable principal ideal. In \cite[Proposition 3.4]{DN}, De Negri has shown that powers of squarefree strongly principal ideals are $k$-strongly stable principal ideals. Therefore, since for all $l$, ${I(P_A)}^l$  are $k$-strongly stable principal ideals, it follows from \cite[Theorem 2.1]{GHP} that $I(P_A)^l$ has a linear resolution for all $l$. Now applying  \cite[Proposition 2.1]{HH1} we have that $\depth({I(P_A)}^l)\geq \depth({I(P_A)}^{l+1})$ for all $l$. In particular, if $\mm \in \Ass(I(P_A)^k)$ then  $\mm\in \Ass(I(P_A)^{k+1})$. This implies that $P_A\in \Ass(I^{k+1})$, as desired.
\end{proof}

\section{When is the maximal ideal a stable associated prime ideal?}
The following theorem give an answer to the question raised in the title of this section. The trivial case that $u=x_1$ will not be considered in the following discussion.
\begin{Theorem}
\label{important}
Let $I\subset S=K[x_1,x_2,\ldots,x_n]$ be squarefree strongly stable principal ideal with Borel generator $u\neq x_1$. Then $\mm\in \Ass(I^k)$ for some $k\geq 1$ if and only if $\min(u)>1$ and $\max(u)=n$.
\end{Theorem}
\begin{proof}
 Let $u=x_{i_1}\cdots x_{i_d}$, with $1\leq i_1< \cdots < i_d\leq n$. First we show that, if $\max(u)<n$ then $\mm\notin \Ass(I^k)$ for all $k>0$. By using \cite[Proposition 3.4]{DN}, $I^k$ is a $k$- strongly stable principal ideal with Borel generator $u^k$.

  Let $B_k(u^k)=\{u_1,\ldots, u_r\}$. We may assume that $u_i>_{lex} u_j$ for $i<j$. Then  by \cite[Section 2]{BEOS}, one has
  \begin{eqnarray}
  \label{colon}
  (u_1,\ldots,u_{i-1}):u_i &=&(\{x_j: 1\leq j\leq \max(u_i)-1\}\\ \nonumber
&-&\{x_j\in \supp(u_i): \deg_{x_j}(u_i)=k\}).
  \end{eqnarray}
 Let $q(I^k)$ be the maximal of the number of generators of $(u_1,\ldots,u_{i-1}):u_i$ for $i=1,\ldots,r$. By \cite[Eq.~1]{HH1} it is known that $\depth(S/I^k)=n-q(I^k)-1$. Since $\max(u)<n$, we also have $\max(u_i)<n$ for all $i$. It follows therefore from (\ref{colon}) that $q(I^k)<n-1$. Hence, $\depth(S/I^k)>0$, and so $\mm\notin \Ass(I^k)$.

Now we show that if $\min(u)=1$, then $\mm\notin \Ass(I^k)$. In this case, $u=x_1x_{i_2}\cdots x_{i_d}$ with $1< i_2< \cdots < i_d\leq n$ and $d\geq 2$, because $u\neq x_1$. Since $\min(u)=1$, we have $\deg_{x_1}(u_i)=k$ for all $i$. It follows therefore from (\ref{colon}) that $q(I^k)<n-1$. Therefore, $\depth(S/I^k)>0$, and hence $\mm\notin \Ass(I^k)$.

Finally suppose that $\max(u)=n$ and  $\min(u)>1$. Then $u=x_{i_1}x_{i_2}\cdots x_{i_{r}} x_{n}$, where $1<i_1<i_2<\cdots<i_{r}\leq n-1$. Let $k>r$. Then the element $v=x_1^{r}x_{i_1}^{k-1}x_{i_2}^{k-1}\cdots x_{i_{r}}^{k-1}x_{n}^{k}$ belongs to $G(I^k)$. Hence there exists an integer $i$ such that $v=u_i$. Therefore formula (\ref{colon}) implies that $q(I^k)=n-1$. As a consequence, we have $\depth(S/I^k)=0$, and hence $\mm\in \Ass(I^k)$.
\end{proof}

Let as before $I=(B_1(u))$ and assume that $\min(u)>1$ and  $\max(u)=n$. Then $u=x_{i_1}x_{i_2}\cdots x_{i_{d-1}}x_n$ with $1<i_1<i_2<\cdots <i_{d-1}<n$. In the following we want to determine the smallest  integer  $k$ for which $\mm \in \Ass(I^k)$. To do this we have to introduce some notation. We write the set $\{i_1,\ldots,i_{d-1},n\}$  in a unique way as a disjoint union of intervals. In other words,
\[
\{i_1,\ldots,i_{d-1},n\}=\bigcup_{j=1}^m[a_j,b_j] \quad \text{with} \quad a_i\leq b_j<a_{j+1}\leq b_{j+1}\quad \text{for}\quad j=1,\ldots,m-1.
\]
The number
\begin{eqnarray}
 \label{local1}
 l_j=\left\{
  \begin{array}{ll}
    b_j-a_{j}+1, & \hbox{if \,  $j<m$;} \\
    n-a_m, & \hbox{if \,  $j=m$.}
  \end{array}
\right.
\end{eqnarray}
is the length of interval $[a_j,b_j]$ for $j<m$, and $l_m$ is the length of the last interval minus one.

Similarly we define the length of the gap intervals
\begin{eqnarray}
 \label{local1}
 k_j=\left\{\begin{array}{ll}
 a_1-1, & \hbox{if \,  $j=1$;} \\
 a_j-b_{j-1}-1, & \hbox{if \,  $j>1$.}
 \end{array}\right.
\end{eqnarray}

\medskip
Let $I\subset S$ be monomial ideal and $P\subset S$ be a monomial prime ideal. Suppose that $P\in \Ass(I^k)$ for some $k$. We denote by $\lambda (P;I)$ the smallest index $k$ for which this is the case, and set $\lambda (P;I)=\infty$ if $P\notin \Ass(I^k)$ for all $k$. The number  $\lambda (P;I)$ is called the {\em index of stability}  of $P$ with respect to $I$.
\begin{Theorem}
\label{lamda1}
Let $I$ be a squarefree strongly stable principal ideal with Borel generator
\[
x_{i_1}x_{i_2}\cdots x_{i_{d-1}} x_{n} \quad \text{with}\quad 1<i_1<i_2<\cdots<i_{d-1}< n.
\]
With the notation introduced in (\ref{local}) and (\ref{local1}) for the length of the intervals and gap intervals  of the sequence $\{i_1,i_2,\ldots,i_{d-1},n\}$, we have
\begin{eqnarray}
\label{lamda}
\lambda(\mm;I)=\max_{j=1,\ldots,m}\Big\{\Big\lceil \frac{l_1+l_2+\cdots +l_j}{k_1+k_2+\cdots +k_j}\Big\rceil +1 \Big\}
\end{eqnarray}
\end{Theorem}
\begin{proof}
By \cite[Equation 1]{HH1}, $\mm \in \Ass(I^k)$ if and only if $q(I^k)=n-1$. So we want to find the smallest integer $k$ for which $q(I^k)=n-1$. In other words, we have to find the smallest integer $k$ for which there exist a monomial $v\in B_k(u^k)$ with the property that $\deg_{x_j}(v)<k$ for $j=1,\ldots, n-1$.

If $v\in B_k(u^k)$ with $\deg_{x_i}(v)<k$ for all $i<n$. Then $v$ must be of the form
\[
v=\prod_{s=1}^m(\prod_{i\in [b_{s-1}+1,a_{s-1}]}x_i^{c_i}\prod_{i\in [a_s,b_s]}x_i^{d_i}),
\]
where all $c_i,d_i<k$, except possibly $d_m$ which may be equal to $k$, where $b_0$ is defined to be $0$, and where
\[
\sum_{s=1}^j\sum_{i\in [b_{s-1}+1,a_{s-1}]}c_i \geq \sum_{s=1}^j\sum_{i\in [a_s,b_s]}(k-d_i)
\]
for all $j=1,\ldots, m$. Since
\[
(k_1+k_2+\cdots +k_j)(k-1)\geq \sum_{s=1}^j\sum_{i\in [b_{s-1}+1,a_{s-1}]}c_i,
\]
and
\[
l_1+l_2+\cdots +l_j\leq \sum_{j=s}^j\sum_{i\in [a_s,b_s]}(k-d_i)
\]
it follows that
\[
(k_1+k_2+\cdots +k_j)(k-1)\geq l_1+l_2+\cdots +l_j.
\]

This shows that $\lambda(\mm;I)\geq \max_{j=1,\ldots,m}\Big\{\Big\lceil \frac{l_1+l_2+\cdots +l_j}{k_1+k_2+\cdots +k_j}\Big\rceil +1 \Big\}$.

In order to prove the opposite inequality, we let $k$ be the maximum on the right hand side of the above inequality. Then there exists $c_i<k$ such that
\[
\sum_{s=1}^j\sum_{i\in [b_{s-1}+1,a_{s-1}]}c_i \geq l_1+l_2+\cdots +l_j
\]
for all $j=1,\ldots,m$.
This implies that
\[
v=\prod_{s=1}^m(\prod_{i\in [b_{s-1}+1,a_{s-1}]}x_i^{c_i}\prod_{i\in [a_s,b_s]}x_i^{k-1})x_n
\]
belongs to $B_k(u^k)$, and shows that $\lambda(\mm ;I)\leq k$.
\end{proof}

\begin{Corollary}
\label{bound}
Let $u\in K[x_1,x_2,\ldots, x_n]$ be a squarefree monomial of degree $d$, and let $I$ be the squarefree strongly stable principal ideal with Borel generator $u$. Assume that $\lambda(\mm;I)<\infty$. Then
\begin{enumerate}
\item[(a)] $\lambda(\mm;I)\leq d$.
\item[(b)] $\lambda(\mm;I)=d$ if and only if $u=x_2x_3\cdots x_dx_n$.

\end{enumerate}
\end{Corollary}

\begin{proof}(a)
 Since $l_1+l_2+\cdots +l_m=d-1$, the maximum value of $\lceil \frac{l_1+l_2+\cdots +l_j}{k_1+k_2+\cdots +k_j}\rceil$ is at most $d-1$. Hence  (\ref{lamda}) implies $\lambda(\mm;I)\leq d$.

(b) Let $u=x_2x_3\cdots x_dx_n$. Then by (\ref{lamda}),
\[
\lambda(\mm;I)=\lceil \frac{l_1}{k_1}\rceil +1=d-1+1=d.
\]

Now let $u\neq x_2x_3\cdots x_dx_n$. Then either $k_1\geq 2$, or there exists a $1<j<m$ such that $k_j\geq 1$. So the maximum value of $\lceil \frac{l_1+l_2+\cdots +l_j}{k_1+k_2+\cdots +k_j}\rceil$ is less than $d-1$, and hence by (\ref{lamda}), $\lambda(\mm;I)< d$.
\end{proof}

\begin{Corollary}
\label{values}
Given an integer $d\geq 2$. Then for any integer $2\leq i\leq d$, there exists a squarefree monomial $u$ of degree $d$ in $2d-i+1$ variables such that $\lambda (\mm;I)=i$ for $I=B_1(u)$.
\end{Corollary}
\begin{proof}
Choose the monomial $u=(\prod_{j=2}^ix_j\prod_{j=1}^{d-i}x_{i+2j})x_{2d-i+1}$.
\end{proof}

\section{The stable set of prime ideals}

Let $I$ be a squarefree strongly stable principal ideal with Borel generator $u=x_{i_1}\cdots x_{i_d}$. In this section we want to determine the subsets $A=\{k_1<\cdots<k_s\}$ of $[n]$, for which the prime ideal $P_A$ belongs to $\Ass^{\infty}(I)$.

We have seen in Theorem~\ref{localization} that $I(P_A)$ is a squarefree strongly stable principal ideal. We denote its Borel generator by $u_A$. We observe that
\begin{equation} \label{obs}
\max(u_A)\geq \max(u_B) \quad \text{for} \quad A\subset B\subset [n].
\end{equation}
For the main result of this section we need the following results.
\begin{Lemma}
\label{notexist}
Let $u=x_{i_1}x_{i_2}\cdots x_{i_d}$ and $A=\{k_1<k_2<\cdots<k_s\}$ with $k_s\leq i_d$. Then the following conditions are equivalent:
\begin{enumerate}
\item[(a)] $\max(u_A)<i_d$.
\item[(b)] $k_{s-j}>i_{d-j-1}$ for some $j\geq 0$.
\item[(c)]  there exist an integer $j$ with $0\leq j \leq s-1$ such that  $l(s-j)\geq d-j$,  where $l(t)=\min\{r\:\; k_t\leq i_r\}$.
\end{enumerate}
\end{Lemma}
\begin{proof} We prove the implication (a) \implies (b) by induction on $d$. The assertion is trivial for $d=1$. Now let $d>1$, and suppose $k_{s-j}\leq i_{d-j-1}$ for all $j\geq 0$. Then we have,
\begin{eqnarray}
\label{ineq1}
k_1\leq i_{d-s},k_2\leq i_{d-s+1},\ldots,k_{s-1}\leq i_{d-2},k_s\leq i_{d-1}.
\end{eqnarray}
 Let $B=\{k_2,k_3,\ldots,k_s\}$, and set $v=u_{A\setminus B}$. Then it follows from (\ref{local}) that $v=x_{i_1}x_{i_2}\cdots \widehat{{x_{i_j}}}\cdots x_{i_d}$, where $j=\min \{r\:\, k_1\leq i_r\}$. Since $k_1\leq i_{d-s}$ it follows that $j\leq d-s$. Therefore the corresponding inequalities (\ref{ineq1}) for $B$ compared with the indices of $v$ hold. Since $\max(v)=i_d$ and $\deg(v)=d-1$, we may apply our induction hypothesis and obtain that $\max(v_B)=i_d$. Since $u_A=v_B$, we conclude that $\max(u_A)=i_d$, a contradiction.

(b) \implies (a):
 By assumption, we have $k_{s-t}>i_{d-t-1}$  for some $t\geq 0$. Then $|\{r\:\,k_{s-t}\leq i_r\}|\leq t+1$. Now we show by induction on $j$ that whenever
 \begin{eqnarray}
 \label{new}
 |\{r\:\,k_{s-j}\leq i_r\}|\leq j+1
 \end{eqnarray}
  for some $j$, then $\max(u_A)< i_d$. If $j=0$, then (\ref{new}) implies $i_r$ is the unique index of $u$ with the property that $k_s\leq i_r$. Therefore it follows from (\ref{local}) that $\max(u_A)<i_d$. Suppose now that for $j>0$, the inequality (\ref{new}) holds. Let $v=u_{\{s-j\}}=x_{i_1}x_{i_2}\cdots x_{i_{d-j-1}}x_{i_{d-j+1}}\cdots x_{i_d}$. Therefore the set of indices in $v$ which are greater or equal to $k_{s-j+1}$ is contained in the set $\{i_{d-j+1},\ldots,i_d\}$, whose cardinality is $j$. Our induction hypothesis implies that $\max(v)<i_d$. Since $\max(u_A)\leq \max(v)$, the assertion follows.

  (b) \iff (c): It follows from the definition of the function $l(t)$ that $l(s-j)\geq d-j$ for some $j\geq 0$, if and only if $k_{s-j}>i_{d-j-1}$.
\end{proof}

Given $u=x_{i_1}x_{i_2}\cdots x_{i_d}$ and $A=\{k_1<k_2<\cdots<k_s\}$. Set $k_0=0$ and $k_{s+1}=n+1$, we introduce the following numbers. Let $f=\max \{r\in[s]_0\:\, k_r+1<k_{r+1}\}$ where $[s]_0=0,\ldots,s$. Then $\max(S_A)=k_{f+1}-1$. Let $g=\max \{t\:\,i_t\leq k_{f+1}-1\}$, and set $B=\{k_{f+1},\ldots,k_s\}$. Since $u_A=(u_B)_{A\setminus B}$, it follows by Lemma~\ref{notexist} that $\max(u_A)=\max(S_A)$ if and only if $l(f-j)<g-j$ for $j=0,\ldots,r-1$. Let $h=\min \{j\:,\ i_{j+1}>i_j+1\}$ and set $i_0=0$. Now we are ready to state our main result.

\begin{Theorem}
\label{main}
Let $I\subset S=K[x_1,x_2,\ldots,x_n]$ be squarefree strongly stable principal ideal with Borel generator $u\neq x_1$. With the notation introduced above the following conditions are equivalent:
\begin{enumerate}
\item[(a)] $P_A\in \Ass^{\infty}(I)$.
\item[(b)] {\em (i)} $\min(u_A)>\min(S_A)$ and
{\em (ii)} $\max(u_A)=\max(S_A)$.
\end{enumerate}
Moreover, {\em(b)(i)} holds if and only if $k_t=t$ for $t=0,\ldots,h$, and {\em (b)(ii)} holds if and only if  $l(f-j)<g-j$ for $j=0,\ldots, f-1$.
\end{Theorem}
 \begin{proof}
 One has $P_A\in \Ass^{\infty}(I)$ if and only if $\mm_A\in \Ass^{\infty}(I(P_A))$. Therefore, the equivalence of  (a) and (b) follows from Theorem~\ref{important}. The remaining statements follow from the definition of the numbers $h$, $f$ and $g$ and from Lemma~\ref{notexist}.
 \end{proof}

 Consider the  following simple, concrete example of a squarefree strongly stable principal ideal $I$ with Borel generator $u=x_1x_3x_4x_5 \in K[x_1,\ldots,x_5]$. We use Theorem~\ref{main} to determine the subsets $A\subset [5]$ for which $P_A\in \Ass^{\infty}(I)$. The last column of the table gives the smallest power of $I^k$ of $I$ with  $P_A\in \Ass(I^k)$.

\vspace{1cm}
\begin{center}
\begin{tabular}{ |c | c | c | c|}
  \hline
  $A$ & $u_A$ & $P_A$ & $\lambda(P_A;I)$ \\\hline
  $\{2,3,4,5\}$ & $x_1$ & $(x_1)$ & 1 \\\hline
  $\{1,2,5\}$ & $x_4$ & $(x_3,x_4)$ & 1 \\\hline
  $\{1,3,4\}$ & $x_5$ & $(x_2,x_5)$ & 1 \\\hline
  $\{1,3,5\}$ & $x_4$ & $(x_2,x_4)$ & 1 \\\hline
  $\{1,4,5\}$ & $x_3$ & $(x_2,x_3)$ & 1 \\\hline
  $\{1,2,3\}$ & $x_5$ & $(x_4,x_5)$ & 1 \\\hline
  $\{1,2,4\}$ & $x_5$ & $(x_3,x_5)$ & 1 \\\hline
  $\{1,2\}$ & $x_4x_5$ & $(x_3,x_4,x_5)$ & 2  \\\hline
  $\{1,3\}$ & $x_4x_5$ & $(x_2,x_4,x_5)$ & 2 \\\hline
  $\{1,4\}$ & $x_3x_5$ & $(x_2,x_3,x_5)$ & 2 \\\hline
  $\{1,5\}$ & $x_3x_4$& $(x_2,x_3,x_4)$ & 2 \\\hline
  $\{1\}$ & $x_3x_4x_5$& $(x_1,x_2,x_3,x_4)$ & 3 \\\hline

\end{tabular}
\end{center}

\end{document}